\documentclass[12pt]{article}
\usepackage{epsfig}
\usepackage{latexsym}
\usepackage{graphicx}
\usepackage{amsmath}
\usepackage{amsfonts}
\usepackage{amssymb}

\usepackage{authblk}

\newtheorem{theorem}{Theorem}

\newtheorem{definition}{Definition}
\newtheorem{lemma}{Lemma}
\newtheorem{corollary}{Corollary}
\newtheorem{conjecture}{Conjecture}

\newtheorem{claim}{Claim}
\newtheorem{observation}{Observation}

\newcommand{\per}{{\rm per}}

\newcommand{\qed}{\hfill\rule{0.5em}{0.809em}}

\newenvironment{proof}{
\par
\noindent {\bf Proof.}\rm}{\mbox{}\hfill\rule{0.5em}{0.809em}\par}

\begin{document}
\title{Every nice graph 
	is   $(1,5)$-choosable}
\author{Xuding Zhu\thanks{Grant Numbers: NSFC  11971438,
		U20A2068. ZJNSFC   LD19A010001. E-mail: xdzhu@zjnu.edu.cn}}

\affil{Department of Mathematics, Zhejiang Normal University, Jinhua 321004, China}


\maketitle

\begin{abstract}

A graph $G=(V,E)$ is total weight $(k,k')$-choosable   if the following holds: For any list assignment  $L$ which assigns to each vertex
$v$   a set $L(v)$ of $k$ real numbers, and assigns to each edge
$e$ a set $L(e)$ of $k'$ real numbers, there is a proper $L$-total weighting, i.e., a map
 $\phi: V  \cup E \to \mathbb{R}$ such that $\phi(z) \in L(z)$ for $z \in V \cup E$,
and $\sum_{e \in E(u)}\phi(e)+\phi(u) \ne \sum_{e \in
E(v)}\phi(e)+\phi(v)$ for every edge $\{u,v\}$. A graph is called nice if it contains no isolated edges. As a strengthening of the famous 1-2-3 conjecture, it was conjectured in [T. Wong and X. Zhu,  Total  weigt choosability of graphs, J. Graph Th. 66 (2011),198-212] that every nice  graph  is total weight $(1,3)$-choosable. The problem whether there is a constant $k$ such that every nice graph   is total weight $(1,k)$-choosable remained open for a decade and was recently solved by Cao [L. Cao, Total weight choosability of graphs: Towards the 1-2-3 conjecture, J. Combin. Th. B, 149(2021), 109-146], who proved that every nice graph is total weight $(1, 17)$-choosable. 
This paper improves this result and proves that every nice graph    is total weight
$(1, 5)$-choosable.
\end{abstract}
{\small \noindent{{\bf Key words: }  Total weight choosability;  
1-2-3 conjecture, inner product, Combinatorial Nullstellensatz, Permanent.}

\section{Introduction}

Assume $G=(V,E)$ is a graph. In the following, we assume $V=\{1,2,\ldots, n\}$, and each edge    is a 2-subset $e=\{i,j\}$ of $V$. 
For $i \in V$, let $E(i)$ be the set of edges incident to $i$. 
 A {\em total weighting} of   $G$ is a mapping
$\phi: V  \cup E \to \mathbb{R}$. A total weighting $\phi$ is {\em proper}
if for any edge $\{i,j\}$ of $G$,
$$\sum_{e \in E(i)}\phi(e) + \phi(i) \ne \sum_{e\in E(j)}\phi(e) + \phi(j).$$
 
A proper total weighting $\phi$ with $\phi(i)=0$ for all vertices $i$ is also called
a {\em vertex colouring edge weighting}.  We say a graph admits a {\em vertex colouring $k$-edge weighting} if there is a vertex colouring edge weighting using weights from $\{1,2,\ldots, k\}$.
 Karo\'{n}ski,  {\L}uczak and  Thomason   \cite{KLT2004} first studied vertex colouring edge weighting of graphs.
 Observe that if $G$ has an isolated edge, then $G$ does not have a vertex colouring edge weighting. 
 Thus we restrict to graphs with no isolated edges, which we call {\em nice graphs}.  
 Karo\'{n}ski,  {\L}uczak and  Thomason   \cite{KLT2004} conjectured that every nice graph $G$   has a vertex colouring $3$-edge weighting.
This conjecture attracted considerable
attention, and is called the 1-2-3 conjecture.
It is not obvious that there is a constant $k$ such that 
every nice graph admits a vertex colouring  $k$-edge weighting.  
 Karo\'{n}ski,  {\L}uczak and  Thomason   \cite{KLT2004} proved that there exists 183 real numbers   so that every nice graph has a vertex colouring edge weighting using the 183 real numbers as weights. Then Addario-Berry,
 Dalal, McDiarmid, Reed and Thomason \cite{Add2007} showed that every nice graph admits a vertex colouring  $30$-edge weighting. The upper bound on $k$ was further reduced to $16$ by 
   Addario-Berry, Dalal and Reed \cite{ADR2008}, and to $13$ by   Wang and Yu \cite{WY2008}. The current known best result on 1-2-3 conjecture   was obtained by 
Kalkowski, Karo\'{n}ski   and Pfender
\cite{KKP10}, who   proved that  every nice graph $G$ admits a vertex colouring 5-edge weighting.

Total weighting of graphs was first studied by Przyby{\l}o and  Wo\'{z}niak   \cite{PW2010}.    They  defined
$\tau(G)$ to be the least integer $k$ such that $G$ has
  a proper total weighting $\phi$ with $\phi(z) \in \{1,2, \ldots, k\}$ for $z \in V \cup E$.
  They  conjectured  that $\tau(G)=2$ for all nonempty graphs $G$.
The best result concerning this conjecture was obtained by  Kalkowski   \cite{K2009}, 
who proved that every graph $G$ has a proper
total weighting $\phi$ with $\phi(i) \in \{1,2\}$ for $i \in V$
and $\phi(e) \in \{1,2,3\}$ for $e \in E(G)$.

The list version of edge weighting  of graphs was introduced by
Bartnicki,   Grytczuk and   Niwczyk in \cite{BGN09}. We say a graph $G$ is {\em edge-weight $k$-choosable} if for any list assignment $L$ which assigns to each edge $e$ a set $L(e)$ of $k$ real numbers as permissible weights, there exists an edge weighting $\phi: E(G) \to \mathbb{R}$ such that for any edge $e=\{i,j\}$,  $\sum_{e \in E(i)}\phi(e) \ne \sum_{e \in E(j)}\phi(e)$.  As a generalization of the 1-2-3 conjecture, Bartnicki,   Grytczuk and   Niwczyk proposed the following conjecture:

\begin{conjecture} 
	\label{edge3choosableconjecture}
	Every nice graph   is edge-weight   $3$-choosable.
\end{conjecture} 

The list
version of total weighting of graphs was introduced independently by
Przyby{\l}o and Wo\'{z}niak \cite{PW2011}  and by Wong and Zhu in
\cite{WZ11}. Suppose $\psi: V  \cup E  \to \mathbb{N}^+$. 
A $\psi$-list assignment of $G$ is a mapping $L$ which
assigns to $z \in V \cup E$ a set $L(z)$ of $\psi(z)$ real
numbers. Given a total list assignment $L$, a proper $L$-total
weighting is a proper total weighting $\phi$ with $\phi(z) \in L(z)$
for all $z \in V \cup E$. We say $G$ is {\em total weight
$\psi$-choosable} ($\psi$-choosable for short) if for any $\psi$-list assignment $L$, there is a
proper $L$-total weighting of $G$. We say $G$ is total weight $(k,k')$-choosable ($(k,k')$-choosable for short)
if $G$ is $\psi$-total weight choosable, where $\psi(i)=k$ for $i
\in V(G)$ and $\psi(e) = k'$ for $e \in E(G)$. 

As strengthenings of the 1-2-3 conjecture  and the 1-2 conjecture, the following conjectures were proposed in \cite{WZ11}.

\begin{conjecture} 
	\label{13choosable conjecture}
	Every nice graph   is    $(1,3)$-choosable.
\end{conjecture} 

\begin{conjecture} 
	\label{22choosable conjecture}
	Every graph  is     $(2,2)$-choosable.
\end{conjecture} 

It follows from the definition that every $(1,k)$-choosable graph is edge-weight $k$-choosable, and it is also easy to verify that a graph $G$ is $(k,1)$-choosable if and only if $G$ is $k$-choosable (i.e., for every list assignment $L$ which assigns to each vertex $v$ a set $L(v)$ of $k$ permissible colours, there is a proper colouring $\phi$ of $G$ with $\phi(v) \in L(v)$ for each vertex $v$). So the concept of  total weight choosability put edge weight choosability and vertex choosability of graphs in a same framework. 

For bipartite graphs, $(1,k)$-choosable is strictly stronger than edge-weight $k$-choosable.   For example, it is easy to verify that a path of length 3 is  edge-weight 2-choosable but not $(1,2)$-choosable. However, for connected non-bipartite graphs $G$, 
$(1,k)$-choosable is equivalent to edge-weight $k$-choosable. Assume $G$ is an  edge-weight $k$-choosable connected non-bipartite graph and $L$ is a $(1,k)$-list assignment of $G$.
We show that $G$ has a proper total $L$-weighting. If $L(i)=\{0\}$ for all vertices $i$, then $G$ has a proper total $L$-weighting.  Assume  $L(i) = \{a\}$ for some non-zero real number $a$. Let $W=(i,e_1,j_1,e_2,j_2, \ldots, e_{2k}, j_{2k}, e_{2k+1}, i)$ be a odd length closed walk containing $i$. Let $L'(i)=\{0\}, L'(e_{2i+1}) = \{s+ a/2: s \in L(e_{2i+1})\}$ and 
$L'(e_{2i}) = \{s- a/2: s \in L(e_{2i})\}$, and let $L'(z) = L(z)$ for remaining vertices and edges $z$.
It is easy to verify that $G$ has a proper total $L$-weighting if and only if $G$ has a proper total $L'$-weighting. By repeating this process, we obtain a $(1,k)$-list assignment $L^*$ with $L^*(i)=\{0\}$ for every vertex $i$, and   $G$ has a proper total $L$-weighting if and only if $G$ has a proper total $L^*$-weighting. As $G$ is edge-weight $k$-choosable, we know that $G$ has a proper total $L^*$-weighting, and hence  $G$ has a proper total $L$-weighting.  

Thus Conjecture \ref{13choosable conjecture} implies Conjecture \ref{edge3choosableconjecture}, and they are equivalent for connected non-bipartite graphs. 

Conjectures \ref{edge3choosableconjecture}, \ref{13choosable conjecture} and \ref{22choosable conjecture}  all remain open and seem to be  difficult. As weakenings of these conjectures, the following conjectures  were proposed in \cite{WZ11}.

\begin{conjecture} 
	\label{kk'choosable conjecture}
	There are  constants $k, k'$ such that
	\begin{enumerate}
		\item [(A)] every graph  is   $(k,k')$-choosable.
		\item[(B)]    every nice graph is  $(1,k')$-choosable.
		\item[(C)]  every graph  is    $(k,2)$-choosable. 
	\end{enumerate} 
\end{conjecture} 

(A), (B), (C) are indeed  three conjectures.  
 Each of (B) and (C) is stronger than (A),   (B) and (C) seem to be independent of each other.

The list version of total weighting of graphs also received considerable attention [6-10,  13-15,17-21,24-27]. 
(A)  was confirmed in  \cite{WZ2012}, where it was shown that every graph is $(2,3)$-choosable.  (B)   remained open for a decade, and was recently confirmed by Cao      \cite{Cao}.  With a novel application of real analysis to this combinatorial problem,  Cao proved  that every nice graph    is $(1,17)$-choosable. 
(C)  remains open. 

In this paper, using the tools developed in \cite{Cao}, 
we  prove the result stated in the title.

\begin{theorem}
	\label{thm-key}
	Every nice graph is $(1,5)$-choosable.
\end{theorem}

The proofs for the $(2,3)$-choosability of all graphs in \cite{WZ2012}, the $(1,17)$-choosability of all nice graphs in \cite{Cao} and the $(1,5)$-choosability of all nice graphs in this paper all use Combinatorial Nullstellensatz, and hence are all about algebraic choosabilities. 

\section{Algebraic total weight choosability}

We denote by $\mathbb{N}$ and $\mathbb{N}^+$ the set of non-negative integers and the set of positive integers, respectively. For   $m,n \in \mathbb{N}^+$, let $\mathbb{C}[x_1, x_2, \ldots, x_n]_m$ be the vector space of homogeneous polynomials of degree $m$ in variables $x_1, \ldots, x_n$ over the   field $\mathbb{C}$ of complex numbers.
Let $M_{m,n}(\mathbb{C})$ be the vector space of $m \times n$ matrices   with entries in $\mathbb{C}$.

For a finite set $E$, let 
$$\mathbb{N}^E = \{K: E \to \mathbb{N}  \}, \text{ and }  \mathbb{N}^E_m = \{K \in \mathbb{N}^E:  \sum_{e \in E} K(e)=m\}.$$ 
For 
$K \in \mathbb{N}^E$, let 
$$x^K = \prod_{e \in E}x_e^{K(e)}.$$
Let 
$$K!=\prod_{e \in E}K(e)!.$$
For $K, K'\in \mathbb{N}^E$, we write $K \le K'$ if $K(e) \le K'(e)$ for each $e \in E$. 

Given a  polynomial $P$, we denote the coefficient of the monomial  $x^K$ in the expansion of $P$ by 
$${\rm coe}(x^K, P).$$ 
 Let $${\rm mon}(P)=\{x^K: {\rm coe}(x^K,P) \ne 0\}.$$  
 
 Given a graph $G=(V,E)$, let 
 $$\tilde{P}_G(\{x_z: z \in V \cup E\}) = \prod_{\{i,j\} \in E, i < j}\left(  \left(\sum_{e \in E(i)} x_e+ x_i\right) - \left(\sum_{e \in E(j)} x_e+ x_j\right)\right).$$
 Assign a real number $\phi(z)$ to the variable $x_z$, and view
 $\phi(z)$ as the weight of $z$.
 Let $\tilde{P}_G( \phi  )$ be the evaluation of the
 polynomial at $x_z = \phi(z)$. Then $\phi$ is a proper total
 weighting of $G$ if and only if $\tilde{P}_G( \phi) \ne 0$. 
 
 It follows from Combinatorial Nullstellensatz that if $\prod_{z \in V \cup E}x_z^{K(z)} \in {\rm mon}(\tilde{P}_G)$ (note that $\tilde{P}_G$ is a homogeneous polynomial, and all the non-vanishing monomials are of the highest degree), and $|L(z)| \ge K(z)+1$ for some $K \in \mathbb{N}^{E \cup V}$, then $G$ has a proper total $L$-weighting. 
 
 \begin{definition}
 	 A  graph is said to be {\em algebraic total weight $(k,k')$-choosable} ({\em algebraic  $(k,k')$-choosable} for short) if $x^K=\prod_{z \in V \cup E}x_z^{K(z)} \in {\rm mon}(\tilde{P}_G)$
 	 for some $K \in \mathbb{N}^{E \cup V}_{|E|}$ with $K(i) < k$ for each vertex $i$ and $K(e) < k'$ for each edge $e$.
 \end{definition}
 
   The goal of this paper is to prove that every nice graph   is algebraic   $(1,5)$-choosable. Thus we restrict to monomials of the form $x^K=\prod_{e \in E}x_e^{K(e)}$ of $\tilde{P}_G$. For this purpose, we omit the variables $x_i$ for $i \in V$ and consider the following polynomial:

 $$P_G(\{x_e: e \in E\}) = \prod_{\{i,j\} \in E, i < j}\left(\sum_{e\in E(i)}x_{e} - \sum_{e\in E(j)}x_{e}\right).$$

\begin{definition}
	\label{def-suff}
	Assume $G=(V,E)$ is a graph and $K \in \mathbb{N}^E$. We say $K$ is {\em sufficient for $G$} if there exists $K' \in \mathbb{N}_{|E|}^E$ such that $K' \le K$ and $x^{K'} \in {\rm mon}(P_G)$. 
\end{definition}

Then we have the following observation.

\begin{observation}
	\label{obs}
	A  graph $G=(V,E)$ is   algebraic  $(1, b+1)$-choosable  if and only if there exists $K  \in \mathbb{N}^E$ such that $K$ is sufficient for $G$ and $K(e) \le b$ for  $e \in E$.
\end{observation}

Given a matrix $A=(a_{ij})_{m \times n}$, define a polynomial 
$$F_A(x_1, \ldots, x_n) = \prod_{i=1}^m \sum_{j=1}^n  a_{ij}x_j.$$
Given a graph $G=(V,E)$,
let 
$C_G=(c_{ee'})_{e,e'\in E}$, where for $e=\{i,j\} \in E, i < j$, 
\[
c_{ee'} = \begin{cases} 1, &\text{ if $e'$ is adjacent with $e$ at $i$}, \cr
-1, &\text{ if $e'$ is adjacent with $e$ at $j$}, \cr
0, &\text{ otherwise.}
\end{cases}
\]
It is easy to verify that $$P_G=F_{C_G}.$$

For a square matrix $A=(a_{ij})_{n \times n}$,   the {\em permanent} $\per(A)$ of $A$ is defined as  
$$\per(A) = \sum_{\sigma}\prod_{i=1}^n a_{ i \sigma(i)},$$ where the summation is over all permutations $\sigma$ of 
$\{1,2,\ldots, n\}$.

For $A \in M_{m,n}(\mathbb{C})$, 
for $K   \in \mathbb{N}^n$ and $K'   \in \mathbb{N}^m$,   $A(K)$ denotes the matrix whose columns consist of $K(i)$ copies of the $i$th column of $A$, and $A[K']$ denotes the matrix whose rows consist  of $K'(i)$ copies of the $i$th row of $A$.

For a graph $G=(V,E)$, let
$A_G=(a_{e i})_{e \in E, i \in V}$, where for $e=\{s,t\} \in E, s < t$, 
\[
a_{ei} = \begin{cases} 1, &\text{ if $i=s$}, \cr
-1, &\text{ if $i=t$}, \cr
0, &\text{ otherwise.}
\end{cases}
\]
and let 
$B_G=(b_{ei})_{e \in E, i \in V}$, where   
\[
b_{ei} = \begin{cases} 1, &\text{ if $i$ is incident to $e$}, \cr
0, &\text{ otherwise.}
\end{cases}
\]
It is known \cite{alontarsi1989,Cao,WZ11,WZ2017} and easy to verify that  for $K   \in \mathbb{N}_{m}^n$,
\begin{eqnarray} 
&&{\rm coe}(x^K, F_A) = \frac{1}{K!} \per(A(K)), \label{eqn1}
\end{eqnarray}
and for $K \in \mathbb{N}_{m}^m$,
\begin{eqnarray} 
&& C_G = A_G B_G^T, \text{ and  } C_G(K) = A_GB_G[K]^T, \label{eqnn1}
\end{eqnarray}
and hence
\begin{eqnarray} 
\label{eqn1b} {\rm coe}(x^K,P_G)= \frac{1}{K!}\per(C_G(K)) =\frac{1}{K!}\per(A_G B_G[K]^T). 
\end{eqnarray}

To calculate the permanent of $A_G B_G[K]^T$, we consider more generally the permanent of $AB^*$ for two matrices $A, B \in M_{m,n}(\mathbb{C})$, where $B^*$ is the conjugate transpose of $B$ (so  $B^*=B^T$ when $B$ is a real matrix). 
 
  We can write $AB^*$ as 
  $$ (\overline{b_{11}} {\rm col}_1(A)+\ldots + \overline{b_{1n}} {\rm col}_n(A), \ldots, \overline{b_{m1}} {\rm col}_1(A)+\ldots + \overline{b_{mn}} {\rm col}_n(A)).$$

  For $\sigma \in  [n]^m$, let $M_{\sigma}$ be the matrix
  $$(\overline{b_{1\sigma(1)}} {\rm col}_{\sigma(1)}(A), \ldots, \overline{b_{m\sigma(m)}} {\rm col}_{\sigma(m)}(A)).$$
  
  As permanent is linear with respect to its columns, 
  we have 
  $$\per(AB^*) = \sum_{\sigma \in  [n]^m} \per(M_{\sigma}).$$ 
  For $K =(k_i)_{i \in [n]} \in \mathbb{N}^n_m$, let 
  $$S(K)= \{\sigma   \in  {[n]}^m: |\sigma^{-1}(i)| = k_i\}.$$
  Note that for $\sigma \in S(K)$, 
  $$\per(M_{\sigma}) = \left(\prod_{j=1}^m \overline{b_{j\sigma(j)}} \right) \per(A(K)).$$
  As $\sum_{\sigma \in S(K)} \prod_{j=1}^m \overline{b_{j\sigma(j)}}=\overline{{\rm coe}(x^K, F_B)} $, we have
  $$\sum_{\sigma \in S(K)} \per(M_{\sigma}) = \left(\sum_{\sigma \in S(K)} \prod_{j=1}^m \overline{b_{j\sigma(j)}} \right)\per(A(K)) =  \overline{{\rm coe}(x^K, F_B)} \per(A(K)).$$
  As $\per(A(K)) = K! {\rm coe}(x^K, F_A)$, we have  
  \begin{eqnarray*}
  	\per(AB^*) 
  	&=& \sum_{K \in \mathbb{N}_m^n} \overline{{\rm coe}(x^K, F_B)} \per(A(K)) = \sum_{K \in \mathbb{N}_m^n} K! \overline{{\rm coe}(x^K, F_B)} {\rm coe}(x^K, F_A).
  \end{eqnarray*} 
  
  Note that $\mathbb{C}[x_1, \ldots, x_n]_m$ is a complex vector space with basis   $\{x^K: K \in \mathbb{N}^n_m \}$.
   For $f(x_1, \ldots, x_n) \in \mathbb{C}[x_1, \ldots, x_n]_m$, $$f(x_1, \ldots,x_n) = \sum_{K \in \mathbb{N}^n_m} {\rm coe}(x^K, f) x^K.$$ So $\{{\rm coe}(x^K, f): K \in \mathbb{N}^n_m\}$ are  the coordinates of $f$ with respect to this basis. 
  We define an inner product in this vector space   as 
  $$\langle f, g\rangle = \sum_{K \in \mathbb{N}^n_m} K! {\rm coe}(x^K,f)\overline{ {\rm coe}(x^K,g)}.$$
  Thus $$\per(AB^*) = \langle F_A, F_B \rangle.$$
  Consequently, 
  
  \begin{eqnarray} 
  \label{eqn1b} {\rm coe}(x^K,P_G)= \frac{1}{K!}\per(C_G(K)) =\frac{1}{K!}\per(A_G B_G[K]^T) =\frac{1}{K!} \langle F_{A_G}, F_{B_G[K]} \rangle. 
  \end{eqnarray}
  
 It follows from the definition that 
 $$F_{A_G}=\prod_{e=\{i,j\}\in E, i < j} (x_i-x_j), \text{ and }   F_{B_G[K]} = \prod_{e=\{i,j\} \in E, i < j} (x_i+x_j)^{K(e)}.$$
 For simplicity,  for  a finite subset $E$ of ${\mathbb{N} \choose 2}$, $K   \in \mathbb{N}^E$, let 
  $$Q_E=\prod_{e=\{i,j\}\in E, i < j} (x_i-x_j), \text{ and }   H_E^K = \prod_{e=\{i,j\} \in E, i < j} (x_i+x_j)^{K(e)}.$$
 So
  $F_{A_G} = Q_E$ and $F_{B_G[K]} = H_E^K$, where $E$ is the edge set of $G$.

  \begin{definition}
  	\label{def-Etheta}	
  	For $K  \in \mathbb{N}^E$, let $W_{E,m}^K$ be the complex linear space spanned by 
  	$$\{H_E^{K'}: K' \le K, K' \in \mathbb{N}_m^E\}.$$
  \end{definition}
  
  It is obvious that there exists $K' \in \mathbb{N}_{|E|}^E$ such that $K' \le K$ and $\langle Q_E, H_E^{K'} \rangle \ne 0$ if and only if there exists $F \in W_{E,|E|}^K$ such that $\langle Q_E, F \rangle \ne 0$.
  
  We shall  frequently use the following easy observation.
  
  \begin{observation}
  	\label{ob1}
  	If $F_i \in W_{E, m_i}^{K_i}$ for $i=1,2$, then 
  	$$F_1F_2 \in W_{E,m_1}^{K_1}W_{E, m_2}^{K_2}=W_{E, m_1+m_2}^{K_1+K_2}.$$
  \end{observation}
  
  Summing up the discussions above, we have the following lemma.
  
  \begin{lemma}
  	Assume $G$ is a nice graph and $K  \in \mathbb{N}^E$. The following are equivalent:
  	\begin{enumerate}
  		\item $K$ is sufficient for $G$.
  		\item $\langle  H_E^{K'}, Q_E \rangle \ne 0$ for some $K' \le K$.
  		\item $\langle F, Q_E\rangle \ne 0$ for some $F \in W_{E, |E|}^K$.
  		\item $\per(C_G(K')) \ne 0$ for some $K' \le K$, $K' \in \mathbb{N}_{|E|}^{E}$.
  	\end{enumerate}
  \end{lemma}

\section{A key lemma}

\begin{definition}
	Assume $G=(V,E)$ is a graph   and $J$ is a subset of $V$.
	We denote by 
	\begin{enumerate}
		\item $E_{J,1}$  the set of non-isolated edges in $G-J$.
		\item $E_{J,2}$ the set of isolated edges in $G-J$.
		\item $E_{J,3}$ the set of edges with exactly  one end vertex in $J$.
		\item $E_{J,4}$ the set of edges with both end vertices in $J$.
	\end{enumerate} 
	For $i=1,2$, let $G_{J,i}$ be the subgraph with edge set $E_{J,i}$ and vertex set $V_{J,i}=V(E_{J,i}) = \cup_{e \in E_{J,i}} e$ (where each edge is a set of two vertices), $G_{J,4}$ has vertex set $J$, $G_{J,3}$ has vertex set $V(G)$ ($G_{J,3}$ and $G_{J,4}$ may contain isolated vertices).
\end{definition}
Note that $G_{J,1},G_{J,2}, G_{J,4}$ are pairwise vertex disjoint, and $E$ is the disjoint union of $E_{J1},E_{J,2},E_{J,3}, E_{J,4}$.

\begin{figure}[!htb]
	\centering
	\includegraphics[scale=0.45]{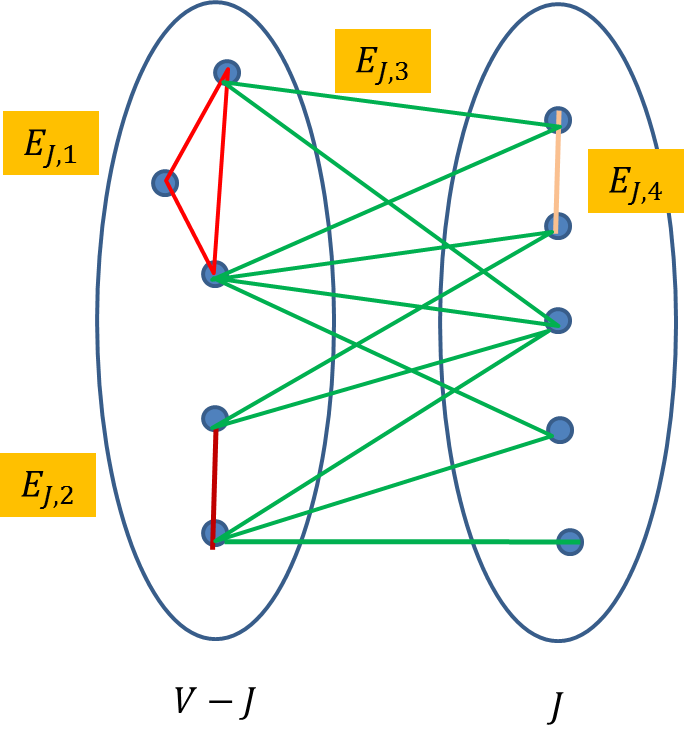}
	\caption{An example of edge sets $E_{J,1},E_{J,2},E_{J,3},E_{J,4}$.   
	}\label{fig-pathfamily0}
\end{figure}
By a subgraph family of $G$, we mean a multi-set $\mathcal{F}$ of subgraphs of $G$. There can be many copies of a same subgraph in $\mathcal{F}$.
If  $\mathcal{F}$ is a subgraph family of $G$ and $k$ is a positive integer, then $k \mathcal{F}$ is the subgraph family of $G$ in which each subgraph in $\mathcal{F}$ occurs $k$ times. If $H$ is a subgraph of $G$ which has multiplicity $m$ in $\mathcal{F}$, then $H$ has multiplicity $km$ in $k \mathcal{F}$.

\begin{definition}
	\label{def-pathfamily}
	Assume $J$ is a subset of $V$.	
	\begin{enumerate}
		\item An {\em $E_{J,2}$-covering family} $\mathcal{C}_{J,2}$ is a family 
		$\mathcal{C}_{J,2} = \{C_e: e \in E_{J,2}\}$ of $|E_{J,2}|$ edges, where for each $e \in E_{J,2}$, $C_e$ is an edge in $E_{J,3}$ adjacent to $e$, i.e., sharing one end vertex with $e$. 
		\item An {\em $E_{J,3}$-covering family} $\mathcal{C}_{J,3}$   is a family  $\mathcal{C}_{J,3}=\{P_1,P_2, \ldots, P_k\}$ of paths whose end vertices are distinct vertices in  $J$, such that the following hold:
		\begin{itemize}
			\item[(i)]  For each pair $i,j$  of distinct vertices in $J$,  there is an even number   of even length paths in $\mathcal{C}_{J,3}$ connecting $i$ and $j$, and an even number   of odd length paths in $\mathcal{C}_{J,3}$ connecting $i$ and $j$.
			\item[(ii)] For each $j \in J$, $$d_{\mathcal{C}_{J,3}}(j) \ge 2 d_{G_{J,3}}(j),$$
			where $d_{\mathcal{C}_{J,3}}(j)$ is the number of paths in $ \mathcal{C}_{J,3}$ with $j$ as an end vertex. 
		\end{itemize} 
		\item 	An {\em $E_{J,4}$-covering family}  $\mathcal{C}_{J,4}$  is a family  $\mathcal{C}_{J,4}=\{C_e: e \in E_{J,4}\}$ of $|E_{J,4}|$ closed walks, 
		where for each $e \in E_{J,4}$, $C_e$ is a   closed walk  of odd length containing $e$.  
		\item A {\em $J$-covering family} $\mathcal{C}$   is the union 
		$$\mathcal{C}= \mathcal{C}_{J,2} \cup \mathcal{C}_{J,3}  \cup \mathcal{C}_{J,4}$$
		of an $E_{J,2}$-covering family, an $E_{J,3}$-covering family and an $E_{J,4}$-covering family. 
	\end{enumerate}
\end{definition}

In the applications below, paths in $\mathcal{C}_{J,3}$ consist of edges from $E_{J,3} \cup E_{J,2}$ and have lengths 2 or 3, and closed walks in $\mathcal{C}_{J,4}$ are triangles consisting of two edges from $E_{J,3}$ and one edge from $ E_{J,4}$. 

\begin{definition}
	For a subgraph $H$ of $G$, let $K_H \in \mathbb{N}^E$ be the characteristic function of $E(H)$, i.e., 
	\[
	K_H(e)=\begin{cases} 1, &\text{ if $e \in E(H)$} \cr
	0, &\text{ if $e \in E(G)-E(H)$.} 
	\end{cases}
	\]
	For a family $\mathcal{H}$ of subgraphs of $G$, 
	$$K_{\mathcal{H}} = \sum_{H \in \mathcal{H}}K_H.$$
\end{definition}

\begin{lemma}[Key Lemma]
	\label{lemma-key}
	Assume $G=(V,E)$ is a   graph, and $J$ is a  subset of $V$,  $\mathcal{C}$ is a $J$-covering family. If  $K \in \mathbb{N}^E$   is sufficient for $G_{J,1}$, then   $  K + K_{\mathcal{C}}$   is sufficient for $G$.    
\end{lemma}

We shall prove Lemma \ref{lemma-key} in Section \ref{sec-prooflemmakey}.
In the next two sections,  we use this lemma to prove Theorem \ref{thm-key}.
First we have the following easy corollary.

\begin{corollary}
	\label{cor-J}
	Assume  $G=(V,E)$ is a   graph,   $J$ is a subset of $V$. 
	If $G_{J,1}$ is algebraic $(1, b+1)$-choosable, and there exists a $J$-covering family $\mathcal{C}$   such that $K_{\mathcal{C}}(e) \le b$ for   $e \in E-E_{J,1}$ and $K_{\mathcal{C}}(e) = 0$ for   $e \in E_{J,1}$, then $G$ is algebraic $(1, b+1)$-choosable. 
\end{corollary}
\begin{proof}
	As $G_{J,1}$ is algebraic $(1, b+1)$-choosable, there exists $K  \in \mathbb{N}^E$  such that $K (e) \le b$ for $e \in E_{J,1}$ and $K (e)=0$ for $e \notin E_{J,1}$  and $K $ is sufficient for $G_{J,1}$.  Let $\mathcal{C}$ be a $J$-covering family such that  $K_{\mathcal{C}}(e) \le b$ for   $e \in E-E_{J,1}$ and $K_{\mathcal{C}}(e) = 0$ for    $e \in E_{J,1}$. By Lemma \ref{lemma-key},   $  K + K_{\mathcal{C}}$ is sufficient for $G$. As $ K(e) + K_{\mathcal{C}}(e) \le b$ for every edge $e$, we conclude that $G$ is algebraic $(1, b+1)$-choosable.
\end{proof}

Note that if $G$ has an isolated edge, then for any subset $J$ of $V$, there is no $J$-covering family. 

 \section{Every nice graph is algebraic  $(1,6)$-choosable}

This section proves  that every nice graph is algebraic $(1,6)$-choosable. The proof of this result is simpler than that of Theorem \ref{thm-key}, but contains the ideas of  how to find a good subset $J$ and construct a required $J$-covering family.  

For a graph $G=(V,E)$  and a subset $J$   of $V$,  $E_{J,i}$ and $G_{J,i}$ are as defined before.

 \begin{definition}	For $i \in V-J$, if  $N_{G_{J,3}}(i) = \{j\}$, then $i$ is called a {\em private neighbour} of $j$ (in $G_{J,3}$).
 \end{definition}

  \begin{definition}
  	We say  $J \subseteq V $ is a {\em good subset}   if $J \ne \emptyset$ and the following hold:
  	\begin{itemize}
  		\item[J1] $G_{J,4}$ has maximum degree at most $1$.
  		\item[J2] $G_{J,3}$ has no isolated edges.
  		\item[J3] Each vertex $j \in J$ has at most one private neighbour.
  		\item[J4] For each edge $e=\{j,j'\} \in E_{J,4}$, there is a vertex $i_e \in V-J$ with $ \{j,j'\} \subseteq  N_{G_{J,3}}(i_e)$ and none of $j,j'$ has a private neighbour. Moreover, $i_e \ne i_{e'}$ for distinct edges $e,e' \in E_{J,4}$.
  		\item[J5] For any edge $e  \in E_{J,2}$, each end vertex $i$ of $e$ has at least one neighbour in $J$. 
  	\end{itemize} 
  \end{definition}

  \begin{lemma}
  	\label{lem-goodset}
  A nice graph	$G$ has a good subset.
  \end{lemma}
  \begin{proof}
  	Choose a maximum independent set $J$ of $G$ so that  $G_{J,3}$ has minimum number of isolated edges. 
  	If $G_{J,3}$ does have an isolated edge $\{i,j\}$, with $i \in V-J, j \in J$, then $J'=(J-\{j\}) \cup \{i\}$ is also a maximum independent set, and $G_{J',3}$ has fewer number of isolated edges, contrary to our choice of $J$. So $G_{J,3}$ has no isolated edges.
  	
  	Let $j_1, j_2, \ldots, j_t$ be vertices in $J$ that have more   than one private neighbours.
  	For $l=1,\ldots, t$, let $S_l$ be the set of private neighbours of $j_l$. Note that 
  	$S_l$ induces a complete graph in $G$, for otherwise, say $i_l,i'_l \in S_l$ are two non-adjacent vertices in $G$, then $(J-\{j_l\}) \cup \{i_l,i'_l\}$ is a larger independent set of $G$, contrary to our choice of $J$. Also $S_1, S_2,  \ldots, S_t$ are pairwise disjoint.  Let $S$ be a maximum independent set of $G[S_1 \cup S_2 \cup \ldots \cup S_t]$. As  each $S_l$ induces a complete graph,  $|S \cap S_l | \le 1$.  Without loss of generality, we may assume that $S= \{i_1,i_2,\ldots, i_s\}$, where $s \le t$ and $i_l \in S_l$. 
  	
  	Let $J' =J \cup S$. Then for $l=1,2,\ldots, s$, $\{i_l,j_l\}$ is an isolated edge in $G_{J',4}$.   Moreover, none of $i_l,j_l$ has a private neighbour. Indeed, since every vertex in $V-J$ has at least one neighbour in $J$, we conclude that $i_l$ has no private neighbour.  On the other hand, vertices in $S_{l} - \{i_l\}$ are adjacent to both $i_l,j_l$, so none of them is a private neighbour of  $j_l$. Other neighbours of $j_l$ in $V-J$ remain to be non-private neighbour of $j_l$. 
  	
  	For $l=s+1,\ldots, t$, each vertex in $S_l$ is adjacent to some vertices in $S \subseteq J'$. So $j_l$ has no private neighbours in $G_{J',3}$.   
  	
  	For $l=1,2,\ldots, s$, let $i'_l \in S_l -\{i_l\}$. Then $\{i_l,j_l\} \subseteq N_{G_{J',3}}(i'_l)$. 
  	
  		As $J$ is a maximum independent set, for any edge $e \in E_{J',2}$ each end vertex $i$ of $e$ has at least one neighbour in $J \subseteq J'$. So $J'$  is a good subset of $V$.
  \end{proof}

\begin{lemma}
	\label{lem-5}
	Assume $G = (V,E)$ is a nice graph and $J$ is a good subset of $V$. Then there exists a $J$-covering family $\mathcal{C}$ for $J$ such that   $K_{\mathcal{C}}(e) \le 5$ for $e \in E-E_{J,1}$ and $K_{\mathcal{C}}(e) = 0$ for $e \in  E_{J,1}$. 
\end{lemma} 
\begin{proof}
	Let $J$ be a good subset of $V$.
	Let $$I = \{i \in V-J: d_{G_{J,3}}(i) \ge 2\}.$$

	For each $i \in I$, we   construct a path family $\mathcal{P}_i$ as follows:   
	 
	 Assume $$N_{G_{J,3}}(i) = \{j_{i,1},j_{i,2}, \ldots, j_{i,t_i}\}.$$    
	 
	 If $i=i_e$ for an edge $e=\{j,j'\} \in E_{J,4}$,   then    we order vertices in $N_{G_{J,3}}(i)$ so that 
	 $j_{i, 1}= j$ and $j_{i,t_{i}} = j'$. Let  $\mathcal{P}_i$   be the path family consisting of the paths $P_{i,l}=(j_{i,l}, i, j_{i,l+1})$.
for $l=1,2,\ldots, t_{i-1}$.

Otherwise let  $\mathcal{P}_i$ be the path family consisting of the paths $P_{i,l}=(j_{i,l}, i, j_{i,l+1})$.
	  for $l=1,2,\ldots, t_i$, where we let $t_i+1 =1$. (See Figure \ref{fig-pathfamily2022}.)
	  
	  	\begin{figure}[!htb]
	  		\centering
	  		\includegraphics[scale=0.65]{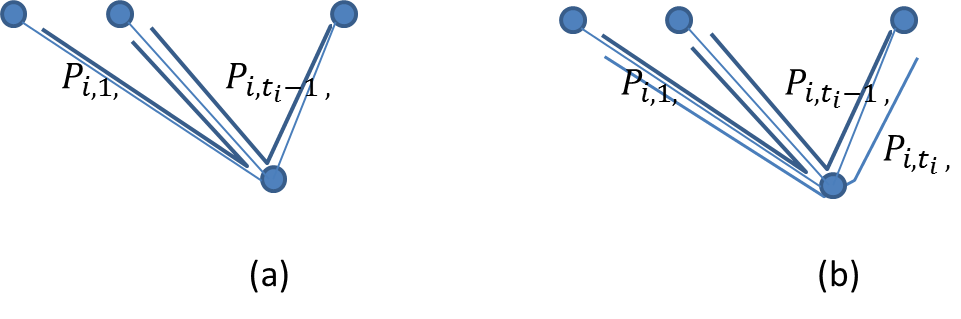}
	  		\caption{The path family $\mathcal{P}_{i}$, (a) for $i=i_e$ for an edge $e \in E_{J,4}$, (b) for other $i \in I$.   
	  	}\label{fig-pathfamily2022}
	  	\end{figure}
	
	Let $$\mathcal{C}'_{J,3}=\cup_{i \in I} \mathcal{P}_i \text{ and }  \mathcal{C}_{J,3} = 2 \mathcal{C}'_{J,3}.$$

	Note that if $t_i=2$ and $i \ne i_e$ for any edge $e \in E_{J,4}$,   then $\mathcal{P}_i$ consists of two copies of the path $(j_{i,1}, i, j_{i,2})$. 

Now we show that $\mathcal{C}_{J,3}$ is an $E_{J,3}$-covering   family. Since   $\mathcal{C}_{J,3} = 2 \mathcal{C}'_{J,3}$,   any $j,j' \in J$ is connected by an even number of even length paths and an even number of odd length paths (indeed, each path in $\mathcal{C}_{J,3}$ constructed above has length $2$, so the number of   odd length paths connecting $j$ and $j'$ is 0). 

	For 
$\mathcal{C}_{J,3}$ to be  an $E_{J,3}$-covering   family, we   need 
$d_{\mathcal{C}_{J,3}}(j ) \ge 2 d_{G_{J,3}}(j)$. This means that for each  contribution of 1 to 
$d_{G_{J,3}}(j)$, there should be a corresponding contribution of $2$ to $d_{\mathcal{C}_{J,3}}(j )$.

Assume $i \in I$ and  $N_{G_{J,3}}(i) = \{j_{i,1},j_{i,2}, \ldots, j_{i,t_i}\}$ ($t_i \ge 2$).
	
	If $i=i_e$ for an edge $e=\{j,j'\} \in E_{J,4}$, then none of $j_{i,1}, j_{i,t_i}$ has a private neighbour. For $l=1, t_i$, the edge $\{i, j_{i,l}\}$   contributes 1 to $d_{G_{J,3}}(j_{i,l})$, and contributes 2 to $d_{\mathcal{C}_{J,3}}(j_{i,l})$; for $2 \le l \le t_i-1$,
	each edge $\{i,j_{i,l}\}$    contributes 1 to $d_{G_{J,3}}(j_{i,l})$, and contributes 4 to $d_{\mathcal{C}_{J,3}}(j_{i,l})$, as there are 4 paths ending with edge  $\{i, j_{i,l}\}$, namely, two copies of  $P_{i,l-1}$ and $P_{i,l}$. 
	 
	If $ i \ne i_e$ for any $e \in E_{J,4}$, then  for $1 \le l \le t_i$, each edge $\{i,j_{i,l}\}$    contributes 1 to $d_{G_{J,3}}(j_{i,l})$, and contributes 4 to $d_{\mathcal{C}_{J,3}}(j_{i,l})$, for the same reason as above.

The extra contribution of 2 to $d_{\mathcal{C}_{J,3}}(j_{i,l})$ from the edge $\{i,j_{i,l}\}$ is used to compensate a possible edge $\{i',j_{i,l}\}$ for which $i'$ is a private neighbour of $j_{i,l}$. 
 Note that if $i'$ is a private neighbour of $j_{i,l}$, then this edge contributes 1 to 
$d_{G_{J,3}}(j_{i,l})$, but there is no  path in $\mathcal{C}_{J,3}$ using the edge $\{i',j_{i,l}\}$. 
So the contribution of the edge $\{i',j_{i,l}\}$ to  $d_{\mathcal{C}_{J,3}}(j_{i,l})$ is $0$.

In case  $i=i_e$ for an edge $e \in E_{J,4}$, then  none of $j_{i,1}, j_{i,t_i}$ has a private neighbour,   for $l=1,t_i$. It suffices for each edge $\{i, j_{i,l}\}$ making contribution 2 to $d_{\mathcal{C}_{J,3}}(j_{i,l} )$. 

Otherwise, 
as each vertex $j \in J$ has at most 1 private neighbour, and when $j$ has a private neighbour, then $j$ has at least one non-private neighbour (since $G_{J,3}$ has no isolated edges), we conclude that $d_{\mathcal{C}_{J,3}}(j ) \ge 2 d_{G_{J,3}}(j)$ for all $j \in J$. 
Hence $\mathcal{C}_{J,3}$ is an $E_{J,3}$-covering  family. 

For any edge $e = \{i,j\}$ with $i \in V-J$ and $j \in J$, $e$ is   contained only in paths in $\mathcal{P}_i$.  By the construction of $\mathcal{P}_i$,  there are at most 2 paths in $\mathcal{P}_i$ that contain the edge $e$, and hence 4 paths in $\mathcal{C}_{J,3}$ that contain the edge $e$.  Moreover, for   $e = \{j,j'\} \in E_{J,4}$, for  $e'_1 = \{i_e,j\}$ and $e'_2 = \{i_e, j'\}$,  there is only one path in $\mathcal{P}_{i_e}$ that contains $e'_1$ and one path in $\mathcal{P}_{i_e}$ that contains $e'_2$. Hence there are only  2 paths in $\mathcal{C}_{J,3}$ that contains each of $e'_1, e'_2$.  

Let $$E'= \{\{i_e,j\}, \{i_e,j'\}: e=\{j,j'\} \in E_{J,4}\}.$$

Therefore 
\[
K_{\mathcal{C}_{J,3}}(e) = \begin{cases} 4, &\text{ if $e \in E_{J,3} -E'$}, \cr
2, &\text{ if $e \in E'$}, \cr
0, & \text{ if $e \notin E_{J,3}$}.
\end{cases}
\]

For $e=\{i_1,i_2\} \in E_{J,2}$, let $C_e =  \{i_1,j_1\}  $  be an arbitrary edge in $ E_{J,3}$ incident to $i_1$. (Since $J$ is a good subset, $i_1$ has a neighbour in $J$, so such an edge $C_e$ exists).  
Then $\mathcal{C}_{J,2} = \{C_e: e\in E_{J,2}\}$ is an $E_{J,2}$-covering family (see Figure \ref{fig-pathfamily0}).

\begin{figure}[!htb]
		\centering
		\includegraphics[scale=0.65]{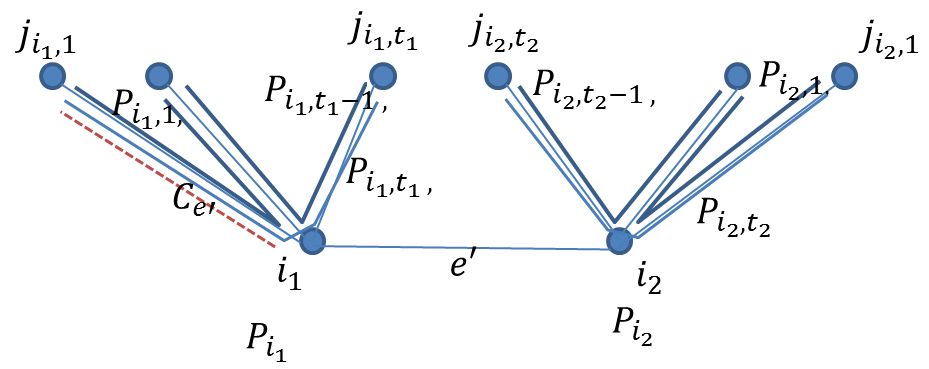}
		\caption{The path families $\mathcal{P}_{i_1}$ and $\mathcal{P}_{i_2}$, and the 
			edge $C_{e'}=\{i_1,j_{i_1,1}\}$ for the edge $e'=\{i_1,i_2\} \in E_{J,2}$. (The figure illustrates the case that $t_1=t_2=3$.)   
			  }\label{fig-pathfamily0}
	\end{figure}
For each edge $e=\{j, j'\} \in E_{J,4}$, let $C_e$ be the closed walk  $(i_e, j,j', i_e)$ (which is a triangle).  Then $\mathcal{C}_{J,4} = \{C_e: e\in E_{J,4}\}$ is an $E_{J,4}$-covering family.

If $e \in E_{J,3}$ is contained in $C_{e'} $ for some  $e'   \in E_{J,4}$,   then    $e \in E'$, and hence 
$$K_{\mathcal{C}_{J,3}}(e) = 2, \text{ and } K_{\mathcal{C}_{J,4}}(e) = 1,  \text{ and } K_{\mathcal{C}_{J,2}}(e) \le 1.$$ 
Hence $K_{\mathcal{C}}(e) \le 4$.
 
If $e \in E_{J,3}-E'$, then $$K_{\mathcal{C}_{J,3}}(e) \le  4, \text{ and } K_{\mathcal{C}_{J,4}}(e) = 0, \text{ and } K_{\mathcal{C}_{J,2}}(e) \le 1.$$ 
Hence
$$K_{\mathcal{C}}(e)   \le 5.$$

Therefore $$K_{\mathcal{C}}(e)   \le 5
 \ \forall e \in E-E_{J,1}, \text{ and } K_{\mathcal{C}}(e)  = 0 \  \forall e \in  E_{J,1}.$$
 This completes the proof of Lemma \ref{lem-5}. 
\end{proof}

The following result is  an immediate consequence of Lemma \ref{lem-5} and Lemma \ref{lemma-key}. 

\begin{theorem}
	\label{thm16}
	Every nice graph is algebraic  $(1,6)$-choosable.
\end{theorem}
\begin{proof}
	The proof is by induction on the number of edges of $G$. 
	We may assume $G$ is connected. If $G$ has no edges,
	 then the theorem holds trivially. 
	 Assume $G=(V,E)$ has $m$ edges and the theorem holds 
	 for any graph with fewer edges. Let $J$ be a good subset of $V$. As $G_{J,4}$ has maximum degree at most 1, we know that $J \ne V(G)$, and $E_{J,4} \ne E(G)$.  
	 Let $\mathcal{C}$ be a   $J$-covering family for which
	  $K_{\mathcal{C}}(e) \le 5$ for   $e \in E-E_{J,1}$ and
	   $K_{\mathcal{C}}(e) =0$ for   $e \in E_{J,1}$. By the induction hypothsis, $G_{J,1}$ is algebraic $(1,6)$-choosable. By Corollary \ref{cor-J}, 
	 $G$ is algebraic  $(1,6)$-choosable. 
\end{proof}

\section{Every nice graph is algebraic  $(1,5)$-choosable}

In this section, we  prove the the following result, which implies Theorem \ref{thm-key}.

\begin{theorem}
	\label{thm-alg15}
	Every nice graph is algebraic $(1,5)$-choosable.
\end{theorem} 
Assume  Theorem \ref{thm-alg15} is not true, and   
\begin{itemize}
	\item  $G=(V,E)$ is a   counterexample with minimum number of vertices. 
	\item $J$ is a good subset of $V$ of minimum size.
\end{itemize}

Following the proof of Theorem \ref{thm16}, it suffices to  construct a $J$-covering family $\mathcal{C}$  so that 
$K_{\mathcal{C}}(e) \le 4$ for   $e\in E-E_{J,1}$  and $K_{\mathcal{C}}(e) =0$ for $e \in E_{J,1}$.

For the $J$-covering family $\mathcal{C}$ constructed in the proof of Lemma \ref{lem-5}, 
 it is easy to verify that  the following   conclusion  holds:
\begin{enumerate}
	\item If $e \in E_{J,2} \cup E_{J,1}$, then  $K_{\mathcal{C}}(e)   = 0.$ 
	\item If $e \in E_{J,4}$, then  $K_{\mathcal{C}}(e)   = 1.$ 
	\item If $e \in E_{J,3}$ is an edge of $C_{e'}$ for some $e' \in E_{J,4}$, then 
	$K_{\mathcal{C}}(e)   \le 4.$ 
	\item If $e \ne C_{e'}$ for any $e' \in E_{J,2}$,  then 
	$K_{\mathcal{C}}(e)   \le 4.$ 
	\item If $e = C_{e'}$ for some edge $e'\in E_{J,2}$ and $e \notin C_{e'}$ for any $e' \in E_{J,4}$,   then
	$K_{\mathcal{C}}(e)   \le 5.$ 
\end{enumerate}

  We will   modify the   construction of the $E_{J,3}$-covering family $\mathcal{C}_{J,3}$ and  choose an $E_{J,2}$-covering family more carefully so that
for the resulting $J$-covering family $\mathcal{C}$, 
 $K_{\mathcal{C}}(e) \le 4$, even if $e = C_{e'}$ for some $e' \in E_{J,2}$.

Assume $e' = \{i_1, i_2\} \in E_{J,2}$, and $e=C_{e'} = \{i_1, j\}$.
In the proof of Theorem \ref{thm16},  the edge $\{i_1,j\} \in E_{J,3}$ may contribute 4 to 
$d_{\mathcal{C}_{J,3}}(j)$. As we observed in the proof of Lemma \ref{lem-5},   only a contribution of 2 is needed to compensate the contribution of $\{i_1,j\}$ to $d_{G_{J,3}}(j)$. The extra contribution of 2 is used to compensate the contribution of a {\em possible edge}  $\{i',j\} \in E_{J,3}$ for which $i'$ is a private neighbour of $j$.  
\begin{enumerate}
	\item If $j$ does not have a private neighbour, then the extra contribution of $2$ is not needed, and possibly we can remove two paths in $\mathcal{P}_{i_1}$ which use the edge $e=\{i_1,j\}$ so that $K_{\mathcal{C}_{J,3}}(e)$ drops to $2$.
	\item If $j$ does have a private neighbour, then  an extra contribution of $2$ to $d_{\mathcal{C}_{J,3}}(j)$ is needed. However, $j$ may have  more than 1 non-private neighbours, and we only need a single contribution of 2 to compensate that private neighbour.
	We may select  only one non-private neighbour $i$ of $j$ to make an extra contribution of 2 to  $d_{\mathcal{C}_{J,3}}(j)$, and for   the other non-private neighbour $i'$ of $j$, again we can remove two paths in $\mathcal{P}_{i'}$ that use the edge $e=\{i',j\}$ so that $K_{\mathcal{C}_{J,3}}(e)$ drops to $2$.
	\item If we can reduce $K_{\mathcal{C}_{J,3}}(e)$ to $2$ for all those edges $e$ for which 
	$K_{\mathcal{C}_{J,2}}(e)=1$, then we have $K_{\mathcal{C}}(e) \le 4$ for all edges $e$.
\end{enumerate}

This is exactly what we shall do.  First we prove some results about  the structure of the graph $G_{J,3}$.

\begin{definition}
	Assume $\{i,j\} \in E_{J,3}$ with $j \in J$. If $j$ has a private neighbour and $i$ is the only non-private neighbour of $j$, then we call $j$ a {\em special neighbour} of $i$.
\end{definition}

\begin{lemma}
	\label{lem-minimum}
	If    $i \in V-J$ has a special neighbour, then    
	$d_{G_{J,3}}(i) \le 2$.  
\end{lemma}
\begin{proof}
	Assume $d_{G_{J,3}}(i) \ge 3$ and $j$ is a special neighbour of $i$.   We show that  
	$J'=J-\{j\}$ is a good subset of $V$. 
	It is easy to see that $E_{J',2} \subseteq E_{J,2}$,  and if $\{i_1, i_2\} \in E_{J',2}$, then 
	each of $i_1,i_2$ has at least one neighbour in $J'$. 
	Moreover, if $\{i',j'\} \in E_{J',3}$, then 
	$i'$ is a private neighbour of $j'$ in $G_{J', 3}$ if and only if $i'$ is a private neighbour of $j'$ in $G_{J, 3}$.  So each of the condition of being a good subset is easily verified.  This contradicts the choice of $J$.  
\end{proof}

\begin{lemma}
	\label{lem-folklore}
	For the graph $G$, the following hold:
	\begin{enumerate}
		\item $G$ does not have a degree 1 vertex adjacent to a degree 2 vertex.
		\item $G$ does not contain two adjacent degree 2 vertices. 
	\end{enumerate}  
\end{lemma}
\begin{proof}
	(1) Assume $G$ has a degree 1 vertex $v$ adjacent to a degree 2 vertex $u$. If $G'=G-\{u,v\}$ has an isolated edge, then it is easy to see that $G-\{u,v\}$ is a single edge and hence $G$ is path of length $3$, which is known to be algebraic $(1,3)$-choosable. Assume $G'$ has no isolated edges. By the minimality of $G$, 
	  $G'$ is algebraic $(1,5)$-choosable. 
	Assume $K' \in \mathbb{N}_{|E(G')|}^{E(G')}$ such that $\per(C_{G'}(K')) \ne 0$ and $  K'(e) \le 4$ for $e \in E(G')$. 
	 Let $e_1=\{u,v\}$ and $e_2=\{u,w\}$, where $w$ is a vertex in $G'$.  Let $K \in \mathbb{N}_{|E(G)|}^E$ be defined as $K(e)=K'(e)$ for $e \in E(G')$ and $K(e_1)=K(e_2) =1$. It is easy to verify that   $\per(C_{G}(K)) \ne 0$. Hence $G$ is algebraic $(1,5)$-choosable. 
	
	(2) Assume $G$ has two adjacent  degree 2 vertices $v$ and $u$. If $G'=G-\{u,v\}$ has an isolated edge, then by using (1), we conclude that $G'$ is either a $C_4$ or a triangle $\{u,v,w\}$ plus a degree 1 vertex adjacent to $w$. It is easy to verify the graph is algebraic $(1,3)$-choosable. Assume $G'$ has no isolated edges. By the minimality of $G$, 
	$G'$ is algebraic $(1,5)$-choosable. 
	Assume $K' \in \mathbb{N}_{|E(G')|}^{E(G')}$ such that $\per(C_{G'}(K')) \ne 0$ and $\max K' \le 4$. 
	Let $e_1=\{u,v\}$,  $e_2=\{u,w\}$, $e_3=\{v,w'\}$, where $w, w'$ are (not necessarily distinct)  vertices in $G'$.  Let $K \in \mathbb{N}_{|E(G)|}^E$ be defined as $K(e)=K'(e)$ for $e \in E(G')$ and $K(e_1)=2, K(e_2) =1$ and $K(e_3)=0$. It is easy to verify that   $\per(C_{G}(K)) \ne 0$. Hence $G$ is algebraic $(1,5)$-choosable.   
\end{proof}

\begin{lemma}
	\label{lem-leaf}
	Assume $e =\{i,i'\} \in E_{J,2}$. If $N_{G_{J,3}}(i)  \ge 2$, then $i$ has at least 2 non-special neighbours in $J$.
\end{lemma}
\begin{proof}
	Assume $i$ has at most one non-special neighour in $J$. Since $i$ has at least 2 neighbours in $J$, we know that $i$ has a special neighbour $j \in J$. Let $i''$ be the private neighbour of $j$. By Lemma \ref{lem-minimum}, $d_{G_{J,3}}(i) = 2$. As $\{i,i'\}$ is an isolated edge in $G-J$, we know that $d_G(i) =3$. Let $j'$ be the other neighbour of $i$ in $J$. Let $G'=G-\{i,j\}$. Then $G'$ and $G$ are as depicted in Figure \ref{fig-ij}.
	
	\begin{figure}[!htb]
		\centering
		\includegraphics[scale=0.65]{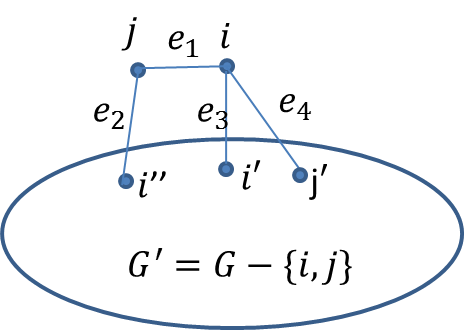}
		\caption{The graph $G$ and $G'$. }\label{fig-ij}
	\end{figure}
	
	If $G'$ has an isolated edge $e$, then $e$ must be incident to $i',i''$ or $j'$. If $e$ is incident to exactly one of $i',i''$ and $j'$, say $e=\{i',w\}$, then $G$ has a degree 1 vertex $w$ adjacent to a degree 2 vertex $i'$, contrary to the discussion above.  
	
	If $e$ connects two of $i',i''$ and $j'$, say $e=\{i',i''\}$, then $G$ has two adjacent degree 2 vertices $i', i''$, again contrary to the discussion above.

	So $G'$ has no isolated edges.
	By the minimality of $G$, $G'$ is algebraic  $(1,5)$-choosable. By (\ref{eqn1b}), there exists $K'=(k'_e)_{e \in E'} \in \mathbb{N}_{|E'|}^{E'}$ with $ K'(e)  \le 4$ for all $e \in E'$ such that 
	$$\per(C_{G'}(K')) \ne 0.$$
	Let $K  \in \mathbb{N}_{|E|}^E$ be defined as 
	\[
	K(e) = \begin{cases} K'(e), &\text{ if $e \in E'$}, \cr
	3, &\text{if $e =e_1$}, \cr
	1, &\text{if $e =e_2$}, \cr
	0, &\text{ if $e \in \{e_3, e_4\}$}.
	\end{cases}
	\] 
	Then 
	\[
	C_G(K)=\left[\begin{array}{c | c |c} 
	C_{G'}(K') &  \star & {\bf 0} \\ \hline
	{\bf 0} & 1 & {\bf 0} \\ \hline
	\star & {\bf 0} & J
	\end{array}
	\right],
	\]  
	where $J$ is a $3 \times 3$ all $1$ matrix, whose rows are indexed by $e_2, e_3, e_4$ and its columns are all indexed by $e_1$. 
	The middle row is indexed by $e_1$ and the middle column is indexed by $e_2$.  
	
	(The $\star$ in the upper-middle of $C_G(K)$ is an $|E'| \times 1$ matrix, and the 
	$\star$ in the 
	left-bottom of $C_G(K)$ is a $3 \times |E'|$ matrix, whose entries are unknown).
	
	Therefore $K \in \mathbb{N}_{|E|}^E$ has $ K(e) \le 4$ for all $e \in E$ and $\per(C_G(K)) \ne 0$. By (\ref{eqn1b}) again, $G$ is algebraic  $(1,5)$-choosable, a contradiction.		
	
	This completes the proof of Lemma \ref{lem-leaf}.
\end{proof}

\begin{definition}
	\label{def-j2cover}
	An {\em   assignment} for $J$ is a mapping $\tau: J \to V-J$ such that 
	for each $j \in J$, $\tau(j)$ is a non-private neighbour of $j$.
\end{definition}

Since $G_{J,3}$ has no isolated edges, and each vertex $j \in J$ has at most one private neighbour, it follows that if $j$ has a private neighbour, then $j$ has a non-private neighbour.
So an  assignment for $J$ always exists.

\begin{definition}
	\label{def-goodassign}
	An assignment   $\tau$   for $J$ is {\em good} if 
	for any $i \in V(G_{J,2})$ with $d_{G_{J,3}}(i) \ge 2$,   $i$ has a neighbour $j \in J$ such that   either $j$ has no priviate neighbour or   $\tau(j) \ne i$. 		
\end{definition}

\begin{lemma}
	\label{lem-goodassign}
	There is a good    assignment for $J$.
\end{lemma}
\begin{proof}
		Let $I=\{i \in V-J: d_{G_{J,3}}(i) \ge 2\}$.   Let $H$ be the edge labeled multi-graph with vertex set $I$ and in which $i,i' \in I$ are joined  by $|N_{G_{J,3}}(i) \cap N_{G_{J,3}}(i')|$ parallel edges, and these edges are labeled by vertices in $N_{G_{J,3}}(i) \cap N_{G_{J,3}}(i')$. For $j \in J$, let $E_j$ be the set of edges in $H$ labeled by $j$ (which induce a clique). Note that $E_j = \emptyset$ if $j$ has only one neighbour $i$ in $V-J$, and in this case, $i$ is a non-private neighbour of $j$. 
		
		For $i,i' \in V(H)$, $d_H(i,i')$ denotes the distance between $i$ and $i'$.

	Let $B$ be a connected component of $H$.

	Assume $B$ contains a vertex $i$ such that 
	either $i \notin V(G_{J,2})$ or $i$
	has a neighbour $j \in J$ which has no private neighbour.
	Let $D$ be an orientation of $B$ in which $i$ is the only sink vertex and for each $j \in  N_{G_{J,3}}(B)$, the sub-digraph $D[E_j]$ of $D$ induced by $E_j$ is transitive tournament.  
	
For    $j \in   N_{G_{J,3}}(B)$, let  $\tau(j)=i'$, where either $E_j = \emptyset$ and $i'$ is the only non-private neighbour of $j$, or
 $E_j \ne \emptyset$, $i'$ is the sink in the transitive tournament $D[E_j]$.

  For any $i' \in B \cap  V(G_{J,2})$ with $d_{G_{J,3}}(i') \ge 2$, 
	either $i'=i$ and hence $i'$ has a neighbour $j \in J$ which has no priviate neighbour, or  
		$i'$ has an out-going edge $e$ in $D$.  In the latter case, assume $e$ is labeled by $j \in J$, then $i'$ is a non-private neighbour of $j$ and 
	$\tau(j) \ne i'$.


	Assume for each   $i \in B \cap V(G_{J,2})$ and
   every neighbour $j \in  J$ of $i$ has a private neighbour.
	By Lemma \ref{lem-leaf}, each vertex $ i \in B$ has at least two non-special neighbours $j \in J$.
	 
	We construct a cycle $C$ in $B$ as follows: Starting from an arbitrary vertex $i_1$, go to the next vertex $i_2$, where $\{i_1, i_2\}$ is an edge   labeled by $j_1$. Assume we arrived at a vertex $i_t$ and the edge $\{i_{t-1}, i_t\}$  is labeled by $j_{t-1}$. Since $i_t$ has at least two non-special neighbours, let $j_t$ be a non-special neighbours of $i_t$ which is distinct from $j_{t-1}$. If $j_t$ is a non-special neighbour of $i_{t'}$ for some $t' \le t-1$, then 
	 $C=(i_{t'}, i_{t'+1}, \ldots, i_t)$ is a cycle in $B$, in which all edges of the cycle are labeled by distinct labels. Otherwise, let   $i_{t+1}$ be any vertex of $B$ so that the edge $\{i_t, i_{t+1}\}$ is labeled by $j_t$. Continuing this process, we will construct a   cycle $C=(i_1,i_2, \ldots, i_q)$ in $B$ (possibly  a 2-cycle), such that  all edges of $C$ are labelled by distinct vertices of $J$.
	  
Assume the edge connecting $i_t$ and $i_{t+1}$ is labeled by $j_t$. Let $\tau(j_t)=i_t$. Orient the remaining edges of $B$ as follows:
Assume $\{i_1,i_2\}$ is an edge in $B$.
If $d_H(i_2,C) < d_H(i_1,C)$ then orient the edge as $(i_1,i_2)$.  If  $d_H(i_2,C) = d_H(i_1,C)$ then orient the edge arbitrarily, subject to the condition that for any $j \in N_{G_{J,3}}(B) - \{j_1,j_2, \ldots, j_q\}$, the sub-digraph induced by edges in $E_j$ is a transitive tournament.  Note that in this orientation,  every   vertex in $B-\{i_1,i_2, \ldots, i_q\}$ has an out-going edge. For  $j \in N_{G_{J,3}}(B) - \{j_1,j_2, \ldots, j_q\}$, let $\tau(j)=i'$, where $E_j = \emptyset$ and $i'$ is the only non-private neighbour of $j$, or
$E_j \ne \emptyset$, $i'$ is the sink in   $D[E_j]$.
   
 For the same reason as in the previous paragraph, for any $i \in B \cap  V(G_{J,2})$,   
	$i$ has a neighbour $j \in J$ such that $i$ is a non-private neighbour of $j$ and $\tau(j) \ne i$. 
	
	Process each component $B$ of $H$ as above, we obtain a good   assignment $\tau$ for $J$.   	  
\end{proof}

\begin{lemma}
	\label{lem-key}
	There is a $J$-covering family $\mathcal{C}$ such that  
 $K_{\mathcal{C}}(e) \le 4$ for $e \in E-E_{J,1}$ and $K_{\mathcal{C}}(e) =0$ for $e \in E_{J,1}$.
\end{lemma}
\begin{proof}
	The $E_{J,4}$-covering family is constructed as in the proof of Lemma \ref{lem-5}. 
	
	The construction of the $E_{J,3}$-covering family $\mathcal{C}_{J,3}$ is modified below.
	
 For $i \in I$ which is not 
	incident to edges in $E_{J,2}$, we construct $\mathcal{P}_i$   as in the proof of Lemma \ref{lem-5}.
	
	Assume $e' = \{i_1,i_2\} \in E_{J,2}$ and $d_{G_{J,3}}(i_1),
	d_{G_{J,3}}(i_2) \ge 2$. 
	Assume
	$N_{G_{J,3}}(i_1)=\{j_{i_1,1}, j_{i_1,2}, \ldots, j_{i_1,t_1}\} $ and $N_{G_{J,3}}(i_2)=\{j_{i_2,1}, j_{i_2,2}, \ldots, j_{i_2,t_2}\}$. 
	Assume $\tau$ is a good assignment for $J$ and $\tau(j_{i_1,1}) \ne i_1$, $\tau(j_{i_2,1}) \ne i_2$.

	\noindent
	{\bf Case 1} 
	$j_{i_1,t_1} = j_{i_2,t_2}$.
	
	Let $\mathcal{P}_{i_1}$ consists of  the paths 
	$P_{i_1,l}=(j_{i_1,l}, i_1,j_{i_1,l+1})$ for $l=1,2,\ldots, t_1-1$. 
	Let $\mathcal{P}_{i_2}$ consists of   the paths 
	$P_{i_2,l}=(j_{i_2,l}, i_2,j_{i_2,l+1})$ for $l=1,2,\ldots, t_2$, 
	where we let $t_2+1=1$ in the indices.   (See Figure \ref{fig-pathfamily2}.)
	
	Let   $e=C_{e'} =  \{i_1, j_1\}$.
	
	\begin{figure}[!htb]
		\centering
		\includegraphics[scale=0.65]{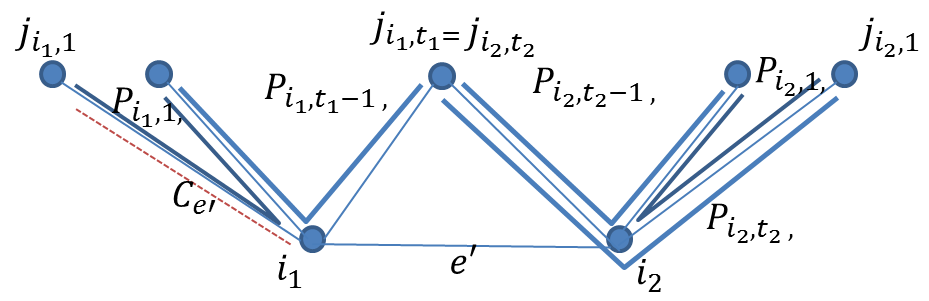}
		\caption{The path families $\mathcal{P}_{i_1}$ and $\mathcal{P}_{i_2}$ for an edge $e'=\{i_1,i_2\} \in E_{J,2}$ with $j_{i_1,t_1} = j_{i_2,t_2}$, and the 
			edge $C_{e'}=\{i_1,j_{i_1,1}\}$ for the edge $e'=\{i_1,i_2\} \in E_{J,2}$.  (The figure illustrates the case that $t_1=t_2=3$.)  }\label{fig-pathfamily2}
	\end{figure}
	
	In Case 1,
	the edge $\{i_1, j_{i_1, t_1}\}$ contributes $1$ to $d_{\mathcal{P}_{i_1} \cup \mathcal{P}_{i_2}}(j_{i_1,t_1})$ (and hence  contributes $2$ to $d_{\mathcal{C}_{J,3}}(j_{i_1,t_1})$)
	and the edge $\{i_2, j_{i_1, t_2}\}$ contributes $2$ to $d_{\mathcal{P}_{i_1} \cup \mathcal{P}_{i_2}}(j_{i_1,t_1})$ (and hence  contributes $4$ to $d_{\mathcal{C}_{J,3}}(j_{i_1,t_1})$). Note that $j_{i_1,t_1} = j_{i_2,t_2}$.

	So there is enough contribution from 
	$\{i_1, j_{i_1, t_1}\}$ and $\{i_2, j_{i_1, t_2}\}$ to $d_{\mathcal{C}_{J,3}}(j_{i_1,t_1})$
	to compensate the contribution to $d_{G_{J,3}}(j_{i_1,t_1})$ from a possible  private neighbour of $j_{i_1,t_1}$.

	\noindent
	{\bf Case 2} $j_{i_1,t_1} \ne j_{i_2,t_2}$.
	
	In this case, instead of constructing 
	$\mathcal{P}_{i_1}$ and $\mathcal{P}_{i_2}$, we construct $\mathcal{P}_{i_1,i_2}$ as follows:

	\begin{itemize}
		\item for $l=1,2,\ldots, t_1-1$,
		$\mathcal{P}_{i_1,i_2}$ contains    the path $P_{i_1, l} = (j_{i_1,l}, i_1, j_{i_1, l+1})$.
		\item For $l=1,2,\ldots, t_2-1$,
		$\mathcal{P}_{i_1,i_2}$ contains  the path $P_{i_2, l} = (j_{i_2,l}, i_2, j_{i_2, l+1})$.
		\item $\mathcal{P}_{i_1,i_2}$ contains  the path $P_{i_1,i_2} = (j_{i_1,t_1}, i_1,i_2, j_{i_2, t_2})$. (See Figure \ref{fig-pathfamily}.)
	\end{itemize} 
	It is easy to verify that   $$K_{\mathcal{P}_{i_1,i_2} }(e)=1, \forall e \in \{\{i_1,j_{i_1,1}\}, \{i_2,j_{i_2,1}\},$$
	and $$K_{\mathcal{P}_{i_1,i_2} }(e)=2, \forall e \in \{\{i_1,j_{i_1,2}\},\ldots,  \{i_1,j_{i_1,t_1}\}, \{i_2,j_{i_2,2}, \ldots, \{i_2,j_{i_2,t_2}\}\}.$$
	
	\begin{figure}[!htb]
		\centering
		\includegraphics[scale=0.65]{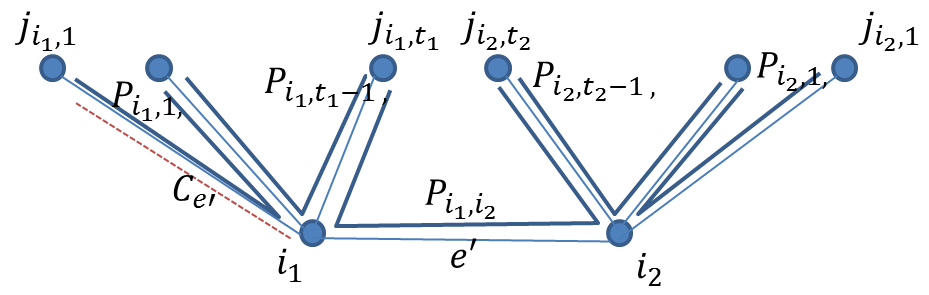}
		\caption{The path family $\mathcal{P}_{i_1,i_2}$ for an edge $e'=\{i_1,i_2\} \in E_{J,2}$, and the 
			edge $C_{e'}=\{i_1,j_{i_1,1}\}$ for the edge $e'=\{i_1,i_2\} \in E_{J,2}$.  (The figure illustrates the case that $t_1=t_2=3$.)  }\label{fig-pathfamily}
	\end{figure}

	Let $$E'_{J,2}= \{  \{i_1,i_2\} \in E_{J,2}: j_{i_1,t_1} \ne j_{i_2,t_2}\}.$$  
	Let $I'= I - V(E'_{J,2})$. 
	
	Let $$\mathcal{C}'_{J,3} = (\cup_{i \in I'} \mathcal{P}_i) \cup (\cup_{\{i_1,i_2\} \in E'_{J,2}}  \mathcal{P}_{i_1,i_2}), \text{ and } 
	\mathcal{C}_{J,3} = 2 \mathcal{C}'_{J,3}.$$

	For each edge $\{i_1,i_2\}$ of $E_{J,2}$, let $C_{e'}= \{i_1,j_{i_1, 1}\}$. Here   vertices in $N_{G_{J,3}}(i_1)$ are labeled as in the discussion above of Case 1 and Case 2.
	By the discussion above, for $e=C_{e'} =  \{i_1, j_1\}$, we have $K_{\mathcal{C}_{J,3}}(e)=2$
	and hence $K_{\mathcal{C}}(e)\le 4$.

If $d_{G_{J,3}}(i_1)=1$ or  $d_{G_{J,3}}(i_2)=1$, say
$N_{G_{J,3}}(i_1)=\{j\}$, then for the $E_{J,3}$-covering family 
$\mathcal{C}_{J,3}$ constructed above, $K_{\mathcal{C}_{J,3}}(e)=0$ for $e = \{i_1, j\}$. Let $C_{e'} =e$. Then $K_{\mathcal{C}}(e)=1$.

 For the $J$-covering family $\mathcal{C}$ constructed above, we have
   $K_{\mathcal{C}}(e) \le 4$ for $e \in E-E_{J,1}$ and $K_{\mathcal{C}}(e) =0$ for $e \in E_{J,1}$.
 This completes the proof of Lemma \ref{lem-key}.
\end{proof}

 Since $G$ is a minimum counterexample to Theorem \ref{thm-key}, $G_{J,1}$ is algebraic $(1,5)$-choosable.
 Following the proof of Theorem \ref{thm16}, using Lemma \ref{lem-key} instead of Lemma \ref{lem-5}, we conclude that $G$ is algebraic $(1,5)$-choosable, contrary to our assumption that   $G$ is a counterexample to Theorem \ref{thm-alg15}. This completes the proof of Theorem \ref{thm-alg15}.
 \qed

  \section{Proof of   Lemma \ref{lemma-key}}
  \label{sec-prooflemmakey}
  
  To prove Lemma \ref{lemma-key}, we need to find $F \in W_{E,|E|}^{K+K_{\mathcal{C}}}$ so that 
  $$\langle F, Q_E \rangle \ne 0.$$
  
  Note that $E$ is the disjoint union of $E_{J,1}, E_{J,2},E_{J,3}$ and $E_{J,4}$. So 
  $$Q_E = Q_{E_{J,1}}Q_{E_{J,2}}Q_{E_{J,3}}Q_{E_{J,4}}.$$
  Let $m_i = |E_{J,i}|$ for $i=1,2,3,4$.
  
  We shall find $$F_1 \in W_{E. m_1}^K, F_2 \in W_{E,m_2}^{K_{\mathcal{C}_{J,2}}}, F_3 \in W_{E,m_3}^{K_{\mathcal{C}_{J,3}}}, F_4 \in W_{E,m_4}^{K_{\mathcal{C}_{J,4}}},$$ and let  
  $F=F_1F_2F_3F_4$. Since $\mathcal{C} = \mathcal{C}_{J,2} \cup \mathcal{C}_{J,3} \cup \mathcal{C}_{J,4}$,  it follows that   $$F =F_1F_2F_3F_4 \in W_{E, |E|}^{K+K_{\mathcal{C}}}.$$
   
For an edge $e = \{i,j\}$ of $G$, let $F_e = x_i+x_j$. 
The polynomial $F_1$ will follow from the assumption of the lemma (explained below).
Each of the polynomials $F_2,F_3,F_4$ will be a product of 
linear combinations of $F_e$'s for the edges in the subgraphs in   $\mathcal{C}_{J, 2},\mathcal{C}_{J, 3}$ and $\mathcal{C}_{J, 4}$, respectively. 

Recall that in the definition of a $(J,3)$-covering family $\mathcal{C}_{J, 3}$, it is required that for each $j \in J$,  $$d_{\mathcal{C}_{J,3}}(j) \ge 2 d_{G_{J,3}}(j).$$
For $j \in J$, let $d_j= d_{\mathcal{C}_{J,3}}(j) - 2 d_{G_{J,3}}(j)$. Let $E'_j =\{  \{i_l,j\}: l = 1,2, \ldots, d_j\}$ (parallel edges are allowed, i.e., $i_l$ need not be distinct vertices ). Let $$E'_{J,3} = E_{J,3} \cup (\cup_{j \in J}E'_j).$$

\begin{lemma}
	\label{lem-add2}
	Assume $E \subseteq E' \subseteq {[n] \choose 2}$ and $K \in \mathbb{N}^{E'}$. If there exists $F \in W_{E, |E'|}^{K}$ for which $\langle F, Q_{E'} \rangle \ne 0$, then there exists 
	$F' \in W_{E, |E|}^{K}$ such that 
	$$\langle F', Q_E \rangle \ne 0.$$
\end{lemma}
\begin{proof}
	Let $G'=(V,E')$, which is obtained from $G$ by adding edges in $E'-E$.
	Assume $|E'| = m$ and $|E'|-|E|=k$. 
	Now $\langle F, Q_{E'} \rangle \ne 0$ implies that 
	there exists $K'  \in \mathbb{N}_m^E$ such that 
	$ K' \le K$ and $\per(C_{G'}(K')) \ne 0$. 
	
	Note that $C_{G'}(K')$ is an $m \times m$ matrix. 
	For each set $I$ of $k$ columns of $C_{G'}(K')$, let  $M_I$ be the $k\times k$ submatrix of $C_{G'}(K')$ consisting of the   columns in $I$ and the rows indexed by edges in $E'-E$. Let 
	$M'_I$ be the $|E|\times |E|$ submatrix of $C_{G'}(K')$ obtained by deleting the rows and columns of $M_I$. We have
	$$\per(C_G(K')) = \sum \per(M_I) \per(M'_I)$$
	where the summation is over all choices $I$ of $k$ columns of $C_G(K')$.
	So $\per(M'_I) \ne 0$ for some choice $I$. 
	The matrix $M'_I$ equals $C_G(K'')$ for some $K'' \le K' \le K$.  So there exists $F' \in W_{E, |E|}^{K}$ such that $\langle F', Q_E \rangle \ne 0.$
\end{proof}

In Lemma \ref{lem-add2}, it is crucial that $F \in W_{E,|E'|}^K$ (not in $W_{E', |E'|}^K$). The added edges add more constraint inequalities, but do not add more variables. 
By using Lemma \ref{lem-add2}, in the following, we assume that for each $j \in J$,   $d_{\mathcal{C}_{J,3}}(j) = 2 d_{G_{J,3}}(j).$

\begin{definition}
	 The polynomials $F_1,F_2,F_3,F_4$ are constructed as follows:
 \begin{enumerate}
 	\item   By our assumption, $K$ is sufficient from $G_{J,1}$. So there exists $F_1 \in W_{E, m_1}^K$ such that $$\langle F_1, Q_{E_{J,1}} \rangle \ne 0.$$
 	
 	\item Recall that for each edge $e \in E_{J,2}$, $C_e$ is an edge in $E_{J,3}$ adjacent to $e$.    Let $$F_2 = \prod_{ e   \in E_{J,2}}F_{C_e} \in W_{E, m_2}^{K_{\mathcal{C}_{J,2}}}.$$
 	
 	\item  For each path $P \in \mathcal{C}_{J,3}$, if the two end vertices of $P$ are $i$ and $j$, with $i < j$, we let $s(P)=i$ and $t(P)=j$. Let $\ell(P)$ be the length of $P$ (i.e., the number of edges in $P$), and for $l=1, \ldots, \ell(P)$, let $e_l(P)$ be the $l$th edge of $P$. Let
 	$$F_3 = \prod_{P \in \mathcal{C}_{J,3}}(x_{s(P)} + (-1)^{\ell(P)-1}x_{t(P)}) =  \prod_{P \in \mathcal{C}_{J,3}} \sum_{l=1}^{\ell(P)} (-1)^{l-1}F_{e_l(P)} \in W_{E, m_3}^{K_{\mathcal{C}_{J,3}}}.$$	
 	
 	\item For each $e=\{j_1,j_2\} \in E_{J,4}$, choose $j_e \in \{j_1, j_2\}$. Let $C_e \in \mathcal{C}_{J,4}$ be the odd cycle containing $e$. For $l=1, \cdots, \ell(C_e)$, let $e_l(C_e)$ be the $l$th edge of $C_e$, and and we choose the starting vertex of $C_e$ so that   $e_1(C_e)$ and $e_{\ell(C_e)}$ are the two edges of $C_e$ incident to $j_e$. Let
 	$$F_4 = \prod_{e \in E_{J,4}}x_{j_e} = \prod_{e \in E_{J,4}} \frac 12   \sum_{l=1}^{\ell(C_e)} (-1)^l F_{e_l(C_e)}  \in W_{E, m_4}^{K_{\mathcal{C}_{J,4}}}.$$
 \end{enumerate}
\end{definition}

It remains to show that 
\begin{eqnarray}
\label{ineq}
 \langle F_1F_2F_3F_4, Q_{E_{J,1}}Q_{E_{J,2}}Q_{E_{J,3}}Q_{E_{J,4}} \rangle \ne 0.
\end{eqnarray}

In the construction of $F_4$, we need to choose $j_e \in \{j_1,j_2\}$ for each edge $e = \{j_1,j_2\} \in E_{J,4}$. The choice is not arbitrary. It suffices to prove that there exists a choice so that   Inequality (\ref{ineq}) holds. 

Observe that $F_1$ and $Q_{E_{J,1}}$ are polynomials in variables $\{x_i: i \in V(G_{J,1})\}$, and none of the other polynomials in (\ref{ineq}) involves variables in $\{x_i: i \in V(G_{J,1})\}$. Therefore 
$$\langle F_1F_2F_3F_4, Q_{E_{J,1}}Q_{E_{J,2}}Q_{E_{J,3}}Q_{E_{J,4}} \rangle = \langle F_1, Q_{E_{J,1}} \rangle \langle F_2F_3F_4, Q_{E_{J,2}}Q_{E_{J,3}}Q_{E_{J,4}} \rangle.$$
By assumption,  $\langle F_1, Q_{E_{J,1}} \rangle \ne 0$. 
So it suffices to show that $\langle F_2F_3F_4, Q_{E_{J,2}}Q_{E_{J,3}}Q_{E_{J,4}} \rangle \ne 0.$

Note that $Q_{E_{J,2}}$ is a polynomial with variables in $\{x_i: i \in V(G_{J,2})\}$. None of $F_3,F_4$ involves variables in $\{x_i: i \in V(G_{J,2})\}$. For $e \in E_{J,2}$, $C_e = \{i_e, j_e\} \in \mathcal{C}_{J,2}$, $i_e \in V(G_{J,2})$ and $j_e \in J$. Thus  
$$F_2 = \prod_{e \in E_{J,2}} x_{i_e} + F'$$
where $F'$ is a polynomial of degree $|E_{J,2}|$ with each monomials contains positive power of some variable $x_j$ for $j \in J$. Since $Q_{E_{J,2}}$ has degree  $|E_{J,2}|$, the contribution of $F'$ to 
$\langle F_2F_3F_4, Q_{E_{J,2}}Q_{E_{J,3}}Q_{E_{J,4}} \rangle$ is zero. Thus 
$$\langle F_2F_3F_4, Q_{E_{J,2}}Q_{E_{J,3}}Q_{E_{J,4}} \rangle = 
\langle \prod_{e \in E_{J,2}} x_{i_e}, Q_{E_{J,2}} \rangle \langle  F_3F_4,  Q_{E_{J,3}}Q_{E_{J,4}} \rangle.$$
It is easy to see that $$\langle \prod_{e \in E_{J,2}} x_{i_e}, Q_{E_{J,2}} \rangle = {\rm coe}(\prod_{e \in E_{J,2}} x_{i_e}, Q_{E_{J,2}} ) = \pm 1.$$
So it remains to show that 
$$\langle  F_3F_4,  Q_{E_{J,3}}Q_{E_{J,4}} \rangle \ne 0.$$

For each edge $e \in E_{J,3}$, let $j_e \in J$ be the end vertex of $e$ in $J$. Then  $$Q_{E_{J,3}} = \prod_{e \in E_{J,3}}x_{j_e} + F',$$ where each monomial of $F'$ has a positive power of some variable $x_i$ with $i \in V(G)-J$. As $F_3F_4$ are polynomials in variables $\{x_j: j \in J\}$, the contribution of $F'$ to $\langle  F_3F_4,  Q_{E_{J,3}}Q_{E_{J,4}} \rangle$ is zero. So 
$$\langle  F_3F_4,  Q_{E_{J,3}}Q_{E_{J,4}} \rangle = \langle  F_3F_4,   Q_{E_{J,4}}\prod_{e \in E_{J,3}}x_{j_e}  \rangle.$$

Assume for $j,j'\in J$, $j< j'$, $\mathcal{C}_{J,3}$ consists of $2t_{jj'}^+$ even length paths connecting $j$ and $j'$, and $2t_{jj'}^-$ odd length paths connecting $j$ and $j'$. 
It follows from the definition that 
$$F_3 = \prod_{j,j' \in J, j < j'} (x_j-x_{j'})^{2t_{jj'}^+}(x_j+x_{j'})^{2t_{jj'}^-}.$$
Since  for each $j \in J$,    $d_{\mathcal{C}_{J,3}}(j) = 2 d_{G_{J,3}}(j)$, it follows that 
$$\prod_{e \in E_{J,3}}x_{j_e} = \prod_{j,j' \in J, j < j'} (x_jx_{j'})^{t_{jj'}^+ + t_{jj'}^-}.$$

  Let $$\phi = \prod_{j,j' \in J, j<j'} (x_j-x_{j'})^{2t_{jj'}^+}(x_j+x_{j'})^{2t_{jj'}^-}, \text{ and } \psi = \prod_{j,j' \in J, j<j'} (x_jx_{j'})^{t_{jj'}^+ + t_{jj'}^-}.$$ 
   Recall that in the definiton of $F_4$, we need to choose $j_e \in \{j_1,j_2\}$ for each edge $e=\{j_1,j_2\}$ of $E_{J,4}$. For any choice, the  polynomial $F_4$ is a monomial of $Q_{E_{J,4}}$, and any monomial of $Q_{E_{J,4}}$ corresponds to $F_4$ determined by one such choice.
    Therefore to prove Lemma \ref{lemma-key}, it suffices to show that
  there is a monomial $x^K$ of $Q_{E_{J,4}}$ such that 
  $$\langle \phi x^K, \psi Q_{E_{J,4}} \rangle \ne 0.$$ This follows from the following lemma.

  \begin{lemma}
  	\label{lem3}
  	For any $R(x) \in \mathbb{C}[x_1, \ldots, x_n]$, there exists $x^K \in {\rm mon}( R(x))$ such that 
  		$$\langle \phi x^K ,  \psi R(x)   \rangle \ne 0.$$
  \end{lemma}

  In the proof of Lemma \ref{lem3}, we need to use another inner product in the complex vector space $\mathbb{C}[x_1, \ldots, x_n]_m$, which is defined as follows:
  
  $$(f,g) = \sum_{K \in \mathbb{N}^n_m} {\rm coe}(x^K, f)\overline{ {\rm coe}(x^K, g)}.$$

  Assume $P=(p_1,\ldots, p_n), Q=(q_1, \ldots, q_n) \in \mathbb{N}^n_m$. It is straightforward to verify that 
  \begin{eqnarray*}
  	(x^P,x^Q) =  (2\pi)^{-n} \int_{\theta_1=0}^{2\pi} \ldots \int_{\theta_n=0}^{2\pi} \prod_{j=1}^n  e^{i p_j\theta_j}    \overline { \prod_{j=1}^n {e^{iq_j \theta_j} }}d\theta_1 \ldots d\theta_n  = \begin{cases} 1, &\text{ if $P=Q$},\cr
  		0, &\text{ otherwise}. 
  	\end{cases}    
  \end{eqnarray*} 
  
  Hence the inner product   $(f,g)$ can be alternately defined as 
  \begin{eqnarray}
  \label{eqn-innerproduct}
  (f,g) = (2\pi)^{-n} \int_{\theta_1=0}^{2\pi} \ldots \int_{\theta_n=0}^{2\pi}f( e^{i \theta_1},\ldots, e^{i\theta_n}) \overline {g({e^{i\theta_1},\ldots, e^{i\theta_n})}}d\theta_1 \ldots d\theta_n.
  \end{eqnarray}
  
  The following lemma follows easily from the definitions of the two inner products.
  
  \begin{lemma}
  	\label{lem-added}
  	Assume $f,g \in \mathbb{C}[x_1, \ldots, x_n]_m$. Let $\tilde{f} \in \mathbb{C}[x_1, \ldots, x_n]_m$ be a polynomial such that for each $x^K \in {\rm mon}(g)$, 
  	${\rm coe}(x^K, \tilde{f}) = \frac{1}{K!} {\rm coe}(x^K, f)$. Then $$\langle \tilde{f},g \rangle = ( f,g).$$
  \end{lemma}

  \noindent
  {\bf Proof of Lemma \ref{lem3}}
  	Assume ${\rm mon}(R(x)) = \{x^{K_i}: i=1,2,\ldots, l\}$ and $R(x) = \sum_{i=1}^l \alpha_i x^{K_i}$. 
  	Assume to the contrary that for each $1 \le i \le l$, 
  	$$\langle \phi(x)x^{K_i}, \psi(x) R(x)\rangle = 0.$$
  	Then for any $(\beta_i)_{ 1 \le i \le l} \in \mathbb{C}^l$,  
  		$$\langle \phi(x)\sum_{i=1}^l \beta_ix^{K_i}, \psi(x) R(x)\rangle = 0.$$
  	Note that for $1 \le i \le l$, $ \psi(x) x^{K_i}$ is also a monomial, which we denote by 
  	$x^{{\tilde K}_i}$. Then $${\rm mon}(\psi(x)R(x)) = \{x^{{\tilde K}_i}: i=1,\ldots, l \}.$$
  	
  \begin{claim}
  	\label{claim1} There exist $ (\beta_i)_{ 1 \le i \le l} \in \mathbb{C}^l$ such that for $1 \le j \le l$, 
  	$${\rm coe}\left(x^{\tilde{K}_j}, \phi(x) \sum_{i=1}^l \beta_ix^{K_i}\right) =\frac{1}{\tilde{K}_j!}{\rm coe}\left(x^{\tilde{K}_j}, \phi(x) {R(x)}\right). $$
  \end{claim} 
  	
  Assume Claim \ref{claim1} is true. By Lemma \ref{lem-added}, $$  \langle \phi(x)\sum_{j=1}^l \beta_jx^{K_j}, \psi(x) R(x)\rangle = (\phi(x) {R(x)}, \psi(x) R(x) ).$$ However, let $T = \sum_{	j < j', j,j' \in J} t_{jj'}^+$, we have 	 
  	\begin{eqnarray*}
  &&	(-1)^{T} 	(\phi(x) {R(x)}, \psi(x)R(x)) \\
  &=& (2\pi)^{-n}  \int_{\theta_1=0}^{2\pi}
  		\ldots  \int_{\theta_n=0}^{2\pi} 	(-1)^{T}  \phi(e^{i\theta_1}, \ldots, e^{i\theta_n}) \overline{\psi(e^{i\theta_1}, \ldots, e^{i\theta_n})} |R(e^{i\theta_1}, \ldots, e^{i\theta_n})|^2d\theta_1 \ldots d\theta_n\\
  			&=& (2\pi)^{-n} 
  			\int_{\theta_1=0}^{2\pi}
  			\ldots  \int_{\theta_n=0}^{2\pi} 
  			\prod_{j,j' \in J, j < j'} \left( 
  			- \frac{ (e^{i\theta_j}-e^{i\theta_{j'}})^2} { e^{i\theta_j}e^{i\theta_{j'}}}\right)^{t_{jj'}^+}  \\
  			&& \left(  \frac{ (e^{i\theta_j} +e^{i\theta_{j'}})^2}{   e^{i\theta_j} e^{i\theta_{j'}}} \right)^{t_{jj'}^-}  
  			 |R(e^{i\theta_1}, \ldots, e^{i\theta_n} )|^2d\theta_1 \ldots d\theta_n\\
  				&=&   (2\pi)^{-n} 
  				\int_{\theta_1=0}^{2\pi}
  				\ldots  \int_{\theta_n=0}^{2\pi} 
  				\prod_{j,j' \in J, j < j'} \left(  2-2\cos (\theta_j-\theta_{j'})\right)^{t_{jj'}^+} \\
  				&& \left(   2+2\cos (\theta_j-\theta_{j'})  \right)^{t_{jj'}^-}  
  				 |R(e^{i\theta_1}, \ldots, e^{i\theta_n} )|^2d\theta_1 \ldots d\theta_n\\
  		& > & 0, 
  	\end{eqnarray*}
  	a contradiction. 
  	
  	It remains to prove Claim \ref{claim1}.
  	Let $A=(a_{ij})_{l \times l}$, where for $1 \le i, j \le l$, 
  	$$a_{ij} =   (    \phi(x) x^{K_i},\psi(x) x^{K_j} ) =  (    \phi(x) x^{K_i},  x^{\tilde{K}_j} ) ={\rm coe}(x^{\tilde{K}_j}, \phi(x) x^{K_i} ).  $$ 
  	Let $b = (\frac{1}{  \tilde{K}_i!}{\rm coe}(x^{\tilde{K}_i}, \phi(x){R(x)}))_{1 \le i \le l}$. Claim \ref{claim1} is equivalent to say that there exists $\beta \in \mathbb{C}^l$ such that  $$A \beta = b.$$ 
  Thus to prove Claim \ref{claim1}, it suffices to show that   $A$ is non-singular. We prove this by showing that $y A  y^* \ne 0$ for any non-zero $y \in \mathbb{C}^l$.      For any non-zero vector $y \in \mathbb{C}^l$,  let   $$F_{y}(x) = \sum_{i=1}^l y_i x^{K_i}.$$  Then
  	\begin{eqnarray*}
  		y A  y^* &=& \sum_{1 \le i,j \le l} y_i \overline{y_j}  (  \phi(x)  x^{K_i},\psi(x) x^{K_j} ) \\
  		&=&      (    \phi(x) F_{y}(x),  \psi(x)  F_{y}(x) )\\ 
  		&=&  (2\pi)^{-n} \int_{\theta_1=0}^{2\pi} \ldots \int_{\theta_n=0}^{2\pi}    \prod_{i,j \in J, i < j}   (e^{i \theta_i} -e^{i\theta_j})^{2t_{ij}^+}  (e^{i\theta_i} +e^{i\theta_j})^{2t_{ij}^-}\\
  		&& \overline{( e^{i \theta_i} e^{i\theta_j})^{t_{ij}^++ t_{ij}^-} } |F_{y}(e^{i\theta_1}, \ldots, e^{i\theta_n} )|^2 d \theta_1 \ldots d \theta_n \\ 
  			&=&   (2\pi)^{-n} \int_{\theta_1=0}^{2\pi} \ldots \int_{\theta_n=0}^{2\pi}    \prod_{i,j \in J, i < j}   
  			\left(  \frac{ (e^{i \theta_i} -e^{i\theta_j})^2} {e^{i \theta_i} e^{i\theta_j}} \right)^{t_{ij}^+}
  			\left( \frac{ (e^{i \theta_i} + e^{i\theta_j})^2} {e^{i \theta_i} e^{i\theta_j}} \right)^{t_{ij}^-}			\\
  			&&  |F_{y}(e^{i\theta_1}, \ldots, e^{i\theta_n} )|^2 d \theta_1 \ldots d \theta_n \\ 
  		&=& (-1)^{T} (2\pi)^{-n} 
  		\int_{\theta_1=0}^{2\pi} \ldots \int_{\theta_n=0}^{2\pi}   
  		 \prod_{i,j \in J, i < j} \left( 2 - 2 \cos (\theta_i-\theta_j)\right)^{t_{ij}^+}\left( 2 + 2 \cos (\theta_i-\theta_j)\right)^{t_{ij}^-} \\ 
  		&&  |F_{y} 
  			(e^{i\theta_1}, 
  			\ldots, e^{i\theta_n} )|^2 
  			d \theta_1 \ldots d \theta_n.
  	\end{eqnarray*}
  	So $	y A  y^* > 0$ if $T$ is even and $	y A  y^* < 0$ if $T$ is odd. 
 In any case, $ 	y A  y^* \ne 0$.
  	This completes the proof of Claim \ref{claim1}, as well as Lemma \ref{lem3}.
  \qed

\end{document}